\documentclass{amsart}

\usepackage[T1]{fontenc}
\usepackage[utf8]{inputenc}
\usepackage{amsmath,amsthm,amsfonts,amscd,amssymb,eucal,latexsym,mathrsfs}
\usepackage{stmaryrd}
\usepackage{enumerate}
\usepackage{hyperref}
\usepackage[all]{xy}
\usepackage{etoolbox}

\usepackage{tikz,tikz-cd}
\usetikzlibrary{shapes.geometric}
\usetikzlibrary{arrows}
\usepackage{tqft}

\usetikzlibrary{positioning}
\usepackage{float}
\usepackage{MnSymbol}
\usetikzlibrary{matrix}

\theoremstyle{plain}
\newtheorem{theorem}{Theorem}[section]
\newtheorem*{theorem*}{Theorem}

\newtheorem{corollary}[theorem]{Corollary}
\newtheorem*{corollary*}{Corollary}

\newtheorem{lemma}[theorem]{Lemma}

\newtheorem{proposition}[theorem]{Proposition}

\theoremstyle{definition}
\newtheorem{remark}[theorem]{Remark}

\newtheorem{definition}[theorem]{Definition}
\newtheorem*{definition*}{Definition}

\usetikzlibrary{calc,graphs}
\usepackage{xcolor}
\usetikzlibrary{arrows,decorations.pathmorphing}

\newcommand{\ZI}{\mathbb{Z}}

\newcommand{\QI}{\mathbb{Q}}

\DeclareMathOperator{\del}{\partial}

\DeclareMathOperator{\Isom}{\mathrm{Isom}}
\DeclareMathOperator{\Aut}{\mathrm{Aut}}

\DeclareMathOperator{\inj}{\hookrightarrow}

\DeclareMathOperator{\surj}{\twoheadrightarrow}

\DeclareMathOperator{\val}{\mathrm{val}}

\newcommand{\Bord}{\mathrm{Bord}}

\newcommand{\PSL}{\mathrm{PSL}}

\DeclareMathOperator{\rk}{\mathrm{rk}}

\newcommand{\sq}{$\frac 7 4$}
\renewcommand{\th}{$\frac 3 2$}

\newcommand{\ts}{\textsection}

\author{Sylvain Barr\'e}
\author{Mika\"el Pichot}
\address{Sylvain Barr\'e, UMR 6205, LMBA,
Université de Bretagne-Sud,
BP 573,
56017, Vannes, France}
\email{Sylvain.Barre@univ-ubs.fr}

\address{Mika\"el Pichot, McGill University, 805 Sherbrooke St W., Montr\'eal, QC H3A 0B9, Canada}
\title{Isomorphisms and parity of complexes of rank \sq}
\email{mikael.pichot@mcgill.ca}

\begin{document}
\begin{abstract}
We study the isomorphism types of simply connected complexes of rank \sq\ using a local invariant called the parity. We show that the parity can be computed explicitly in certain constructions arising from surgery.  
\end{abstract}

\maketitle

\section{Introduction}

A complex of rank \sq\ is a 2-complex with triangle faces, whose links are isomorphic to the Moebius-Kantor graph.

\begin{figure}[H]
 \includegraphics[width=4cm]{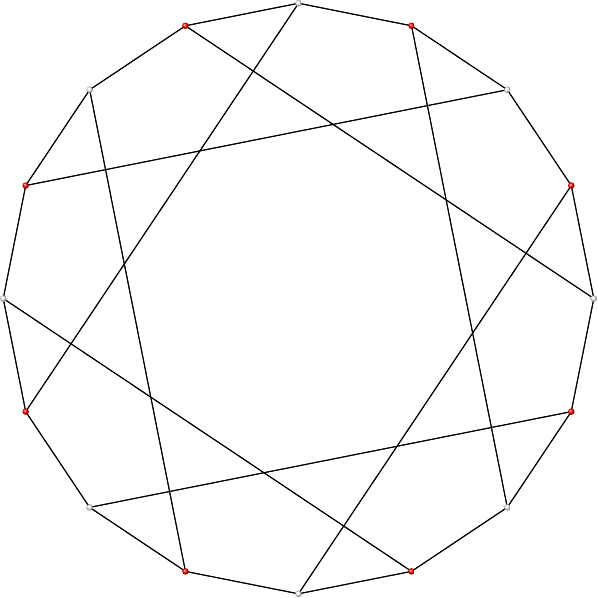} 
 \caption{The Moebius--Kantor graph}
 \end{figure}

\noindent In this paper we  continue our investigations of complexes of rank \sq, which was begun in  \cite{rd}. We study the isomorphism type of simply connected complexes of rank \sq, and aim to prove that 
\begin{enumerate}
\item[] (a) there are infinitely isomorphism types with a cocompact automorphism group
\item[] (b) there are uncountably isomorphism types, in general. 
\end{enumerate}

These statements should be compared to similar ones, for Euclidean Tits buildings, especially in the rank 2 case.

The second statement (b) is reminiscent of  the ``free constructions'' of Ronan and Tits in  \cite{tits1981local,ronan1986construction,ronan1987building} for buildings of type $\tilde A_2$, which can be used to show the existence of uncountably many such buildings.
 The main tool to establish free constructions is a prescription theorem of a local invariant of the complex. For buildings of type $\tilde A_2$, Ronan  has proved  such a theorem in \cite{ronan1986construction}, showing that one can prescribe the residues at vertices arbitrarily, which can be provided by any set of projective planes of fixed order $q$.  

Furthermore, in the case $q=2$, Tits has proved in \cite{tits1988spheres} that there are precisely two isomorphism types of spheres of radius 2, and these spheres of radius 2 can also be used as a local invariant by \cite{barre2000immeubles}.  

The use of  spheres of radius 2 is not very practical in our situation. Our results will in fact imply that there are at least 174163 isomorphism types of such spheres. We shall  rely instead on a local invariant which makes use of the intermediate rank properties of complexes of rank \sq, and in particular, how the roots of rank $2$ (and those of rank $\neq 2$) are organized in the complex. The invariant is defined in \ts\ref{S - rank parity} and called the parity.  A free prescription theorem for the parity  is established in \ts\ref{S- Prescription}.

The first statement (a) is also true for Euclidean buildings. A way to prove it, which is in fact the only way we know of, even in the rank 2 case, is to use Euclidean buildings associated with non isomorphic local fields. For example, if $(K_r)_r$ is a sequence of totally ramified finite extensions of $\QI_p$ which are pairwise nonisomorphic, then (by the classical results of Tits \cite{tits1974buildings}) the Euclidean Tits buildings $X_r$ of $\PSL_d(K_r)$ will be also pairwise nonisomorphic.  

For complexes of rank \sq, constructions using local fields are not available. However, an alternative approach for building such complexes was proposed in \cite{surgery},  using surgery theory. This construction provides compact complexes of rank \sq, and therefore simply connected ones with cocompact automorphism groups, taking universal covers. The question addressed in (a), which was raised in a recent talk by one of us, is to distinguish these complexes up to isomorphism. The problem is that different constructions may lead to isomorphic universal covers.
In \ts\ref{S - surgery construtions}, we show that the parity can be computed explicitly in these constructions, and that it provides a way to distinguish these complexes up to isomorphism. 

Note that infinitely many of the Euclidean buildings $X_r$ defined in the previous paragraph coincide on arbitrary large balls, since there are only finitely many isomorphism types of balls of radius $n$ in this case. In order to prove (a), we must in fact show that a similar construction, where the complexes coincide on arbitrarily large balls, can be done in the case of complexes of rank \sq\ (with cocompact automorphism group). This is because the existence of infinitely many isomorphism types of complexes of rank \sq\ with cocompact automorphism group, as stated in (a), implies the same statement where furthermore, the complexes coincide on arbitrary large finite balls. (Another way to phrase this is to use the space of simply connected complexes of rank \sq, an analog to the space of triangle buildings $E_q$ of \cite{E1,E2}, which is  a compact space.) Therefore, we might as well look directly for such complexes, which is what we shall do. 

We note however that, for complexes of rank \sq, the quotient by the automorphism group in (a) has to be arbitrary large, because the automorphism group is uniformly discrete. An analogous statement for buildings of type $\tilde A_2$ is not known to hold: questions regarding the existence or number of buildings, whose quotient by their automorphism group has a prescribed, e.g., arbitrarily large,  number of vertices, are generally open. This is one of our motivations  to look for new constructions techniques. 

The following is a further analogy between the rank $\frac 7 4$ constructions and the local field constructions for $X_r$. The buildings $X_r$ can be distinguished up to isomorphism, because the field $K_r$ can be recovered abstractly, by the results of Tits \cite{tits1974buildings}, from the building at infinity. However, it is not clear a priori how distant from each other these complexes are in the space of buildings (say in the $\tilde A_2$ case).  A similar phenomenon is also occurs in our rank \sq\ constructions. It is not clear a priori how distant from each other the complexes we build will be in the space of complexes of rank \sq, since there might a priori exist ``exotic'' isomorphisms between the spheres of large radius which we have not detected. Thus, we only prove an injectivity result, see the proof of Theorem \ref{T - pairwise non isomorphic cover}, which explains that the final statement is not more explicit than the one put forward in (a) above. 

\bigskip

\textbf{Acknowledgements.} The second author was partially supported by an NSERC discovery grant and a JSPS award.

\section{The parity invariant}\label{S - rank parity}

We first define a metric (and simplicial) invariant for complexes of rank \sq, called the parity (Def.\  \ref{D  - rank parity}).

Let $X$ be a complex of rank \sq. For convenience, we shall call 2-triangle a simplicial equilateral triangle in $X$ whose sides have length 2. Every 2-triangle contains 4 faces which are themselves (small) equilateral triangles.

Let $x\in X$ be a vertex and $L_x$ the link at $x$ (namely, the sphere of small radius around $x$ in $X$, endowed with the angular metric). Recall that we call \emph{root} at $x$ a metric embedding $\alpha\colon [0,\pi]\inj L_x$ in the link of $x$, such that $\alpha(0)$ is a link vertex. Every root has a rank $\rk(\alpha)$ which is a rational number in $[1,2]$. It is defined by
\[
\rk(\alpha)=1+{N(\alpha)\over q_\alpha}\tag{$\dagger$}
\] 
where 
\[
N(\alpha):=|\{\beta\in \Phi_x\mid \alpha\neq \beta, \alpha(0)=\beta(0), \alpha(\pi)=\beta(\pi)\}|,
\]
writing $\Phi_x$ for the set of roots at $x$ and, for a root $\alpha$, $q_\alpha$ for the valency of $\alpha(0)$ minus 1:
\[
q_\alpha:=\val(\alpha(0))-1.
\]  
We refer to \cite[\ts 4]{chambers} for more details on this definition. In a complex of rank $7/4$, the rank of a root can be 2 or \th.

Every side in a 2-triangle defines two roots at every vertex $x$ of its center (small) triangle, which are of the form $\alpha(t)$ and $\alpha(\pi-t)$ for $t\in [0,\pi]$ for some $\alpha\in \Phi_x$. These two roots have the same rank, which we call the rank of the corresponding side of the 2-triangle.

If $T_1$, $T_2$ are two 2-triangles, we say that $T_1$ and $T_2$ are in \emph{branching configuration} if the intersection $T_1\cap T_2$ contains 3 triangles. In this case, we call \emph{branching permutation}, and we write $T_1\to T_2$, the transformation  that fixes the intersection $T_1\cap T_2$ pointwise, and takes the free triangle in $T_1$ isometrically to the free triangle in  $T_2$. Such a transformation is involutive. The rank of the two sides of $T_1$, which are not included in the fixed set $T_1\cap T_2$, must change under such a transformation:

\begin{lemma}\label{L - rank parity lemma} Let $X$ be a complex of rank \sq. Let $T_1, T_2$ be two 2-triangles of $X$ in branching configuration. 
If $\alpha$ is the root of a side of $T_1$ and $s\colon T_1\to T_2$ is a branching permutation, then the corresponding root transformation $\alpha \mapsto s(\alpha)$ permutes the  values of the rank: 
\[
\begin{cases}\frac 3 2\mapsto 2\\ 2\mapsto \frac 3 2 \end{cases}
\]
unless $s(\alpha)=\alpha$. 
\end{lemma}

\begin{proof}
Consider such a root $\alpha$ at a vertex 
$x$. In the link $L_x$, every isometric embedding $\alpha\colon [0,\frac{2\pi}3]\inj L_x$, where $\alpha(0)$ is a vertex, admits exactly two isometric extensions $\alpha\colon [0,\pi]\inj L_x$ into a root. One of them has rank \th, the other  rank 2. Hence the two values of the rank are permuted by a branching permutation.
\end{proof}

Using Lemma \ref{L - rank parity lemma} one can define a metric invariant of the complex  taking values in $\ZI/2\ZI$, and attached to every face of $X$.

\begin{definition}\label{D  - rank parity}
The \emph{parity} of a triangle face $t$ in a complex of rank \sq\ is the parity of the number of roots of rank 2 in a 2-triangle $T$ in which $t$ embeds as the centerpiece. 
\end{definition} 

It follows from Lemma \ref{L - rank parity lemma} that this is a well defined map
\[
 X^{2}\to \ZI/2\ZI
\]
where $X^{2}$ denotes the set of 2-faces of the 2-complex $X$. We call this map the parity map.
It is clearly an invariant of isomorphism:    

\begin{lemma}
Simplicial isomorphisms between complexes of rank \sq\ preserve the parity of faces.  
\end{lemma}

In particular, the automorphism group $\Aut(X)$ of a complex $X$ of rank \sq\ acts in a parity preserving way.

\begin{definition}
We will say that a complex of rank \sq\ is even (resp.\ odd) if  its faces are even (resp.\ odd). 
\end{definition}

In general, one might expect (compare \ts\ref{S- Prescription}) that a complex of rank \sq\  has mixed parity.

\section{Explicit computations}\label{S- Explicit}

In \cite[\ts 4]{rd} we found 13 complexes of rank \sq. The parity map can be computed explicitly in these examples, which provides new information on these complexes. We explain how in this section. 

The simplest case to consider is that of $X_0:=\tilde V_0$. The complex can be described by its fundamental set of triangles as follows:
\[
V_0:=[[1,2,6],[2,3,7],[3,4,8],[4,5,1],[5,6,2],[6,7,3],[7,8,4],[8,1,5]].
\]
It is the simplest case because the automorphism group is face transitive, and we know in advance that the complex $X_0$ is either odd or even.  

\begin{proposition}\label{P - V0 odd}
$X_0$ is odd. 
\end{proposition}

\begin{proof}
Consider for instance the face $[1,2,6]$. It is contained in $8(=2^3)$ 2-triangles. If the face $[1,2,6]$ is oriented counterclockwise, one of these 2-triangles is given by the faces $[2,3,7]$,  $[6,7,3]$, and $[4,5,1]$ in this order. Labelling the link vertices by a signed integer $\pm k$, where $k\in \{1,\ldots, 8\}$, corresponding to an incoming ($+k$) or a outgoing ($-k$) edge with label $k$ in $X_0$, this gives us three roots at the vertices of $[1,2,6]$
\[
\alpha_1\colon 3\to -6\to 2\to-3
\]
\[
\alpha_2\colon 5\to-1\to6\to-7
\]
\[
\alpha_3\colon 7\to-2\to1\to-4
\]
corresponding to the sides of the given 2-triangle. A direct computation in the labelled link  using $(\dagger)$ shows that 
\[
\rk(\alpha_1)=\rk(\alpha_3)=\frac 3 2
\]
\[
\rk(\alpha_2)=2.
\]
where the representation of the link and its labelling is given by 

\begin{figure}[H]
\begin{tikzpicture}[shift ={(0.0,0.0)},scale = 2.0]
\tikzstyle{every node}=[font=\small]

\node (v1) at (1.0,0.0){$1$};
\node (v2) at (0.92,0.38){$-2$};
\node (v3) at (0.71,0.71){$6$};
\node (v4) at (0.38,0.92){$-7$};
\node (v5) at (-0.0,1.0){$3$};
\node (v6) at (-0.38,0.92){$-4$};
\node (v7) at (-0.71,0.71){$8$};
\node (v8) at (-0.92,0.38){$-1$};
\node (v9) at (-1.0,-0.0){$5$};
\node (v10) at (-0.92,-0.38){$-6$};
\node (v11) at (-0.71,-0.71){$2$};
\node (v12) at (-0.38,-0.92){$-3$};
\node (v13) at (0.0,-1.0){$7$};
\node (v14) at (0.38,-0.92){$-8$};
\node (v15) at (0.71,-0.71){$4$};
\node (v16) at (0.92,-0.38){$-5$};
\draw[solid,thin,color=black,-] (v1) -- (v2);
\draw[solid,thin,color=black,-] (v2) -- (v3);
\draw[solid,thin,color=black,-] (v3) -- (v4);
\draw[solid,thin,color=black,-] (v4) -- (v5);
\draw[solid,thin,color=black,-] (v5) -- (v6);
\draw[solid,thin,color=black,-] (v6) -- (v7);
\draw[solid,thin,color=black,-] (v7) -- (v8);
\draw[solid,thin,color=black,-] (v8) -- (v9);
\draw[solid,thin,color=black,-] (v9) -- (v10);
\draw[solid,thin,color=black,-] (v10) -- (v11);
\draw[solid,thin,color=black,-] (v11) -- (v12);
\draw[solid,thin,color=black,-] (v12) -- (v13);
\draw[solid,thin,color=black,-] (v13) -- (v14);
\draw[solid,thin,color=black,-] (v14) -- (v15);
\draw[solid,thin,color=black,-] (v15) -- (v16);
\draw[solid,thin,color=black,-] (v16) -- (v1);
\draw[solid,thin,color=black,-] (v1) -- (v6);
\draw[solid,thin,color=black,-] (v3) -- (v8);
\draw[solid,thin,color=black,-] (v5) -- (v10);
\draw[solid,thin,color=black,-] (v7) -- (v12);
\draw[solid,thin,color=black,-] (v9) -- (v14);
\draw[solid,thin,color=black,-] (v11) -- (v16);
\draw[solid,thin,color=black,-] (v13) -- (v2);
\draw[solid,thin,color=black,-] (v15) -- (v4);

\end{tikzpicture}
\end{figure}

\noindent This proves that  $X_0$ is odd. 
\end{proof}

This proof works for every face in any complex of rank \sq. Similar computations, for example, will lead to the following statement.

\begin{proposition}\label{P - 74 even}
The following complexes of rank \sq\   are even:
\begin{align*}
V^1_0=[[1, 2, 3], [1, 4, 5], [1, 6, 4], [2, 6, 8], [2, 8, 5], [3, 6, 7], [3, 7, 5], [4, 8, 7]]\\
V^2_0=[[1, 2, 3], [1, 4, 5], [1, 6, 7], [2, 4, 6], [2, 8, 5], [3, 6, 8], [3, 7, 5], [4, 8, 7]]\\
\check V_0^2=[[1, 2, 3], [1, 4, 5], [1, 6, 7], [2, 6, 4], [2, 8, 5], [3, 6, 8], [3, 7, 5], [4, 8, 7]]\\
V_4^1=[[1, 1, 5], [2, 2, 5], [3, 3, 6], [4, 4, 6], [1, 3, 8],  [2, 7, 4],  [5, 8, 7], [6, 7, 8]]
\end{align*}
\end{proposition}

The classification given in \ts4 of \cite{rd} contains eight other complexes of rank \sq. These complexes are of mixed parity.  The details of the parity map are given in an appendix.

\section{Constructing universal covers}\label{S - surgery construtions}

Our goal in this section is to use the parity to construct compact complexes of rank \sq\ with non isomorphic universal covers.

\begin{theorem}\label{T - pairwise non isomorphic cover}
There exists infinitely many compact complexes of rank \sq\ with pairwise non isomorphic universal covers.   
\end{theorem}

An idea is to construct complexes with ``different proportions'' of even  faces, which will therefore have non isomorphic universal covers. A direct implementation of this requires to assign to a complex of rank \sq, in a computable way, an average number of the even faces it contains. In view of the discussion in the introduction, we shall rather use  the following invariant, which measures the `radius of evenness' of the complex $X$.

\begin{definition} Let $X$ be a simply connected complex of rank \sq. We let $e(X)$ be the largest integer $n$ such that there exists a vertex $x\in X$ for which the  ball $B(x,n)$ of radius $n$ is even. 
\end{definition}

Clearly, if $e(X)\neq e(Y)$,  the complexes $X$ and $Y$ and not isomorphic. To prove the theorem, we will construct a sequence $X_n$ of universal covers on which $e$ is injective.

\begin{lemma}\label{L - e finite}
Let $X$ be a simply connected complex of rank \sq\ with cocompact isometry group. Then $X$ is even if and only if $e(X)=\infty$.
\end{lemma}

\begin{proof}
It is clear that $e(X)=\infty$ if $X$ is even. Conversely, assume that $e(X)=\infty$. Then there exists a sequence $(x_n)$ of vertices of $X$ and a sequence $r_n\to \infty$ of radii, such that the ball $B(x_n,r_n)$ in $X$ are even. Since $\Isom(X)$ has a compact fundamental set, we may find a vertex $x$ in $X$ such that for every $n$, the balls  $B(x,r_n)$ in $X$ are even. Thus, $X$ is even.
\end{proof}

In order to find an injectivity set for $e$ we will use the surgery construction of \cite{surgery}. In fact, this is the only construction of (infinitely many) compact complexes of rank \sq\ that we are aware of at this moment.

Let us start with a brief review of these constructions for groups of rank \sq. The surgery is described by a category $\Bord_{\frac 7 4}$ whose arrows are called the group cobordisms.  These objects in this category are called the collars.  Both the objects and the morphisms in $\Bord_{\frac 7 4}$ correspond to 2-dimensional complexes. 

The following 2-complex defines a collar in $\Bord_{\frac 7 4}$. It was used in \cite{randomsurgery} to define a model of random groups of rank \sq.
\[
(x,a,d),\ (y,c,d),\ (z,c,b),\ (x',d,a),\ (y',b,a),\ (z',b,c)
\]
 This 2-complex, which we will denote $C$, is obtained from a set of
six oriented equilateral triangles, with labeled edges, by identifying the edges respecting the labels and the orientations. It is not hard to check that this is an object in $\Bord_{\frac 7 4}$; in fact this collar is included in the ST lemma of \cite[\ts 8]{surgery}.

We will use two arrows 
\[
X_{00},Y_{00}\colon C\to C
\] 
in $\Bord_{\frac 7 4}$, keeping the notation of \cite[\ts 2]{randomsurgery} for consistency. As 2-complexes, these arrows have the following respective presentations 
\[
(x,a,d),(y,c,d), (z,c,b),\ (1,1,2),\ (2,a',d'),(4,c',d'),(3,c',b')
\]
\[
(4,d,a),(3,b,a), (2,b,c),\ (1,3,4),\ (x',d',a'),(y',b',a'),(z',b',c')
\]
and
\[
(x,a,d),(y,c,d), (z,c,b),\ (1,2,3),\ (4,a',d'),(2,c',d'),(1,c',b')
\]
\[
(1,d,a),(3,b,a), (4,b,c),\ (2,4,3),\ (x',d',a'),(y',b',a'),(z',b',c')
\]
described in \cite[Remark 2.8]{randomsurgery}. 
 As group cobordisms, when viewed as arrows $C\to C$ in the category $\Bord_{\frac 7 4}$, these complexes $Z$ are endowed with maps $L_Z\colon C\to X$ and $R_Z\colon C\to C$, which  in the case of $Z=X_{00}$ and $Z=Y_{00}$ are the obvious ones.

As in \ts\ref{S- Explicit}, Prop.\ \ref{P - V0 odd}, a direct computation shows that:

\begin{lemma} The triangles $(1,1,2)$ and $(1,3,4)$ (resp.\ $(1,2,3)$ and $(2,3,4)$) are even in  $X_{00}$ (resp.\ in  $Y_{00}$).
\end{lemma}

Note that the cobordisms themselves are not  complexes of rank \sq, since they have a non trivial boundary. However, since the triangles referred to in the above lemma belong to the core, the link involved is indeed the Moebius--Kantor graph, and the parity makes sense in this case.

On the other hand, the parity of the collar faces is  not determined a priori by the cobordism alone, since it may depend on composition. We shall use this fact to construct complexes with different parity ratios.

\begin{lemma}\label{L - even central}
In the composition $Y_{00}Y_{00}$, the central collar is even.     
\end{lemma}

\begin{proof}
Indeed, the $n$-fold composition $\underbrace{Y_{00}\cdots Y_{00}}_n$ of the cobordism $Y_{00}$ leads to finite covers of the $V_0^1$, which is an even complex, as we pointed out in Prop.\ \ref{P - 74 even}.
\end{proof}

On the other hand we have:

\begin{lemma}\label{L - product odd}
The central collar in the composition $X_{00}Y_{00}$ contains an odd face.
\end{lemma}

\begin{proof}
Let us show that the face $(4,c',d')$ is odd. We can relabel the second cobordism $Y_{00}$ as follows
\[
(2,a',d'),(4,c',d'), (3,c',b'),\ (11,12,13),\ (14,a'',d''),(12,c'',d''),(11,c'',b'')
\]
\[
(11,d',a'),(13,b',a'), (14,b',c'),\ (12,14,13),\ (x'',d'',a''),(y'',b'',a''),(z'',b'',c'')
\]
viewing it as embedded in the composition $X_{00}Y_{00}$. A 2-triangle containing this face is described, in counterclockwise order, by the 3 faces $(1,3,4)$, $(3,c',b')$, and $(2,a',d')$, which gives us three roots 
\[
\alpha_1\colon a'\to -d'\to c'\to -b'
\]
\[
\alpha_2\colon 3\to -c'\to 4\to -1
\]
\[
\alpha_3\colon 3\to -4\to d'\to -2
\]
of which we have to compute the rank. 

The root $\alpha_1$ belongs to the following link in $X_{00}Y_{00}$:

\begin{center}
\begin{tikzpicture}[shift ={(0.0,0.0)},scale = 2.0]
\tikzstyle{every node}=[font=\small]

\node (v1) at (1.0,0.0){$6$};
\node (v2) at (0.92,0.38){$-14$};
\node (v3) at (0.71,0.71){$12$};
\node (v4) at (0.38,0.92){$-7$};
\node (v5) at (-0.0,1.0){$8$};
\node (v6) at (-0.38,0.92){$-5$};
\node (v7) at (-0.71,0.71){$14$};
\node (v8) at (-0.92,0.38){$-13$};
\node (v9) at (-1.0,-0.0){$a'$};
\node (v10) at (-0.92,-0.38){$-11$};
\node (v11) at (-0.71,-0.71){$13$};
\node (v12) at (-0.38,-0.92){$-b'$};
\node (v13) at (0.0,-1.0){$c'$};
\node (v14) at (0.38,-0.92){$-d'$};
\node (v15) at (0.71,-0.71){$11$};
\node (v16) at (0.92,-0.38){$-12$};
\draw[solid,thin,color=black,-] (v1) -- (v2);
\draw[solid,thin,color=black,-] (v2) -- (v3);
\draw[solid,thin,color=black,-] (v3) -- (v4);
\draw[solid,thin,color=black,-] (v4) -- (v5);
\draw[solid,thin,color=black,-] (v5) -- (v6);
\draw[solid,thin,color=black,-] (v6) -- (v7);
\draw[solid,thin,color=black,-] (v7) -- (v8);
\draw[solid,thin,color=black,-] (v8) -- (v9);
\draw[solid,thin,color=black,-] (v9) -- (v10);
\draw[solid,thin,color=black,-] (v10) -- (v11);
\draw[solid,thin,color=black,-] (v11) -- (v12);
\draw[solid,thin,color=black,-] (v12) -- (v13);
\draw[solid,thin,color=black,-] (v13) -- (v14);
\draw[solid,thin,color=black,-] (v14) -- (v15);
\draw[solid,thin,color=black,-] (v15) -- (v16);
\draw[solid,thin,color=black,-] (v16) -- (v1);
\draw[solid,thin,color=black,-] (v1) -- (v6);
\draw[solid,thin,color=black,-] (v3) -- (v8);
\draw[solid,thin,color=black,-] (v5) -- (v10);
\draw[solid,thin,color=black,-] (v7) -- (v12);
\draw[solid,thin,color=black,-] (v9) -- (v14);
\draw[solid,thin,color=black,-] (v11) -- (v16);
\draw[solid,thin,color=black,-] (v13) -- (v2);
\draw[solid,thin,color=black,-] (v15) -- (v4);

\end{tikzpicture}
\end{center}
Therefore it is of rank 2.

The roots $\alpha_2$ and $\alpha_3$ belong to the following link:

\begin{center}
\begin{tikzpicture}[shift ={(0.0,0.0)},scale = 2.0]
\tikzstyle{every node}=[font=\small]

\node (v1) at (1.0,0.0){$1$};
\node (v2) at (0.92,0.38){$-1$};
\node (v3) at (0.71,0.71){$4$};
\node (v4) at (0.38,0.92){$-6$};
\node (v5) at (-0.0,1.0){$5$};
\node (v6) at (-0.38,0.92){$-3$};
\node (v7) at (-0.71,0.71){$b'$};
\node (v8) at (-0.92,0.38){$-c'$};
\node (v9) at (-1.0,-0.0){$3$};
\node (v10) at (-0.92,-0.38){$-4$};
\node (v11) at (-0.71,-0.71){$d'$};
\node (v12) at (-0.38,-0.92){$-a'$};
\node (v13) at (0.0,-1.0){$2$};
\node (v14) at (0.38,-0.92){$-8$};
\node (v15) at (0.71,-0.71){$7$};
\node (v16) at (0.92,-0.38){$-2$};
\draw[solid,thin,color=black,-] (v1) -- (v2);
\draw[solid,thin,color=black,-] (v2) -- (v3);
\draw[solid,thin,color=black,-] (v3) -- (v4);
\draw[solid,thin,color=black,-] (v4) -- (v5);
\draw[solid,thin,color=black,-] (v5) -- (v6);
\draw[solid,thin,color=black,-] (v6) -- (v7);
\draw[solid,thin,color=black,-] (v7) -- (v8);
\draw[solid,thin,color=black,-] (v8) -- (v9);
\draw[solid,thin,color=black,-] (v9) -- (v10);
\draw[solid,thin,color=black,-] (v10) -- (v11);
\draw[solid,thin,color=black,-] (v11) -- (v12);
\draw[solid,thin,color=black,-] (v12) -- (v13);
\draw[solid,thin,color=black,-] (v13) -- (v14);
\draw[solid,thin,color=black,-] (v14) -- (v15);
\draw[solid,thin,color=black,-] (v15) -- (v16);
\draw[solid,thin,color=black,-] (v16) -- (v1);
\draw[solid,thin,color=black,-] (v1) -- (v6);
\draw[solid,thin,color=black,-] (v3) -- (v8);
\draw[solid,thin,color=black,-] (v5) -- (v10);
\draw[solid,thin,color=black,-] (v7) -- (v12);
\draw[solid,thin,color=black,-] (v9) -- (v14);
\draw[solid,thin,color=black,-] (v11) -- (v16);
\draw[solid,thin,color=black,-] (v13) -- (v2);
\draw[solid,thin,color=black,-] (v15) -- (v4);

\end{tikzpicture}
\end{center}
Both of them are of rank $\frac 32$. 

This proves that the face is odd.  
\end{proof}

\begin{proof}[Proof of Theorem \ref{T - pairwise non isomorphic cover}]
Consider the compact complexe of rank \sq\
\[
V_n:=X_{00}\underbrace{Y_{00}\cdots Y_{00}}_{2n}/\sim
\]
where $\sim$ identifies the two extremal copies of $C$ in the composition $X_{00}Y_{00}\cdots Y_{00}$. Let $X_n$ denote the universal cover of $V_n$. Since the projection $X_n\surj V_n$ reduces the distances, it is clear that if $x_n$ denotes a lift the core vertex in the $n^\text{th}$ cobordism $Y_{00}$ in $V_n$, then  by Lemma \ref{L - even central} the ball of radius $n$ in $X_n$ is even. It follows that $e(X_n)\geq n$. By Lemma \ref{L - e finite} and Lemma \ref{L - product odd}, we have $e(X_n)<\infty$. Therefore, we can find a subsequence of $(V_n)_n$ on which $e$ is injective.
\end{proof}

\begin{remark} 1) By their definition, the groups $\pi_1(V_n)$ in  Theorem \ref{T - pairwise non isomorphic cover} is accessible by surgery.

2) The proof shows that the universal cover of $V_0^1$ is not determined locally among the complexes of rank \sq, namely  one can find compact complexes of rank  \sq\  with universal cover distinct from $X_0^1$, but which coincide with it on arbitrary large balls. Another way to phrase this is to say that $V_0^1$ is an accumulation point in the space of complexes of rank \sq. 

3) 
Lemma \ref{L - even central} shows that the universal covers constructed in the present section always contains a definite amount of even faces, which can never be reduced to zero.
On the other hand, we have an example $X_0=\tilde V_0$ of a complex with no even face (Prop.\ \ref{P - V0 odd}). In the forthcoming section we show how to construct (simply connected) complexes of rank \sq\ with an arbitrary parity map.
\end{remark}

\section{Free constructions}\label{S- Prescription}

J.\ Tits observed that certain free constructions of (e.g.) Euclidean buildings can be given in the rank 2 case, in analogy with the construction of free projective planes. In fact, it was shown by M.\ Ronan in \cite{ronan1986construction} that these  Euclidean buildings (in the $\tilde A_2$ case)  can all be obtained by such free constructions, and furthermore that a prescription theorem  holds, to the effect that the residues at the vertices can be taken to run over any set of projective plane of the given order. Thus, there is  complete freedom in the constructions of buildings of type $\tilde A_2$.  We \cite{ronan1986construction,ronan1987building} for details on these results, and to \cite[\ts 2]{barre2000immeubles} for a different free prescription theorem along these lines, for buildings of type $\tilde A_2$ and order 2.

Our goal in this section is to show that complexes of rank \sq\ behave similarly in this respect. We  prove a similar prescription theorem, replacing the projective planes by a prescription theorem for the parity map.

\begin{theorem}\label{T - Free constructions}
There are free constructions of simply connected complexes of rank \sq\ with a given parity map. In particular, there exist uncountably many pairwise non isomorphic simply connected complexes of rank \sq. 
\end{theorem}

The free construction is (as in the above references) by induction, starting with a ball $B_1$ in a complex of rank \sq\ and extending successively the balls $B_1\subset B_2\subset\cdots$, and setting $X:=\bigcup B_n$, where we have to show that the parity can be chosen freely. This is done in the following lemma, from which the theorem follows easily.

\begin{lemma}
Let $B_n$ be a ball of radius $n$ in a complex of rank \sq, and let $S_n:=B_n\setminus  (B_{n-1})^\circ$ denote the simplicial sphere of radius $n$. Let $p\colon S_n\to\ZI/2\ZI$ be a map defined on the 2-skeleton. Then there exists a ball $B_{n+1}$ in a complex of rank \sq, containing $B_n$, and such that the parity of the faces of $S_n$ in $B_{n+1}$ is given by $p$.
\end{lemma}

\begin{proof}
Let $f$ be a face in $S_n$. Then $f\cap \del B_n$ is either a point, an edge, or a pair of adjacent edges (where $\del B_n$ denotes the topological boundary).

If $f\cap \del B_n$ is a point $x$, then $f$ is adjacent to 4 faces in $S_n$ whose intersections with $\del B_n$ is an edge. Let $T$ be a 2-triangle containing the face $f$ and let $\alpha$ be a root corresponding to $T$ at $x$. We can choose to embed $\alpha$ in the Moebius--Kantor graph $L_x$ in such a way that its rank is $\frac 3 2$ or 2, and therefore can freely decide the parity of $f$ in the construction, which can be chosen to be $p(f)$. Furthermore, this operation only determines the following subgraph of  $L_x$.

\begin{center}
\begin{tikzpicture}[shift ={(0.0,0.0)},scale = 2.0]
\tikzstyle{every node}=[font=\small]

\coordinate (v1) at (0.0,0.0);
\coordinate (v2) at (1,0);
\coordinate (v3) at (-0.5,0.5);
\coordinate (v4) at (0.1,.9);
\coordinate (v5) at (.9,.9);
\coordinate (v6) at (1.5,.5);
\coordinate (v7) at (0,1.1);
\coordinate (v8) at (1,1.1);
\coordinate (w3) at (-0.5,-0.5);
\coordinate (w4) at (0.1,-.9);
\coordinate (w5) at (.9,-.9);
\coordinate (w6) at (1.5,-.5);
\coordinate (w7) at (0,-1.1);
\coordinate (w8) at (1,-1.1);

\draw (v2) node[above] {$u$} ;
\draw (v4) node[above] {$v$};
\draw (v7) node[above] {$w$} ;

\path (v1) edge[solid,thin,color=black,-] node[above]  {$f$} (v2);
\draw[solid,thin,color=black,-] (v1) -- (v3);
\draw[solid,thin,color=black,-] (v3) -- (v4);
\draw[solid,thin,color=black,-] (v4) -- (v5);
\draw[solid,thin,color=black,-] (v5) -- (v6);
\draw[solid,thin,color=black,-] (v2) -- (v6);
\draw[solid,thin,color=black,-] (v3) -- (v7);
\draw[solid,thin,color=black,-] (v7) -- (v8);
\draw[solid,thin,color=black,-] (v8) -- (v6);
\draw[solid,thin,color=black,-] (v1) -- (w3);
\draw[solid,thin,color=black,-] (w3) -- (w4);
\draw[solid,thin,color=black,-] (w4) -- (w5);
\draw[solid,thin,color=black,-] (w5) -- (w6);
\draw[solid,thin,color=black,-] (v2) -- (w6);
\draw[solid,thin,color=black,-] (w3) -- (w7);
\draw[solid,thin,color=black,-] (w7) -- (w8);
\draw[solid,thin,color=black,-] (w8) -- (w6);

\end{tikzpicture}

\end{center}

In this graph the edge labelled $f$ corresponds to the centerpiece of the root $\alpha$, which extends on both side, in a direction depending on its rank. The origin and extremity of $\alpha
$ belong to the two faces in $T$ adjacent to $f$ and intersecting the boundary. Let $g$ be one of these two faces and $\beta$ be a root in $L_x$ corresponding to $g$. Note that $g$ is such that $g\cap \del B_n$ is a single edge. We have to show that the parity of $g$ can be chosen according to $p$. By symmetry, one may assume that $\beta(0) = u$ and $\beta(\pi)=v$ or $w$. Lemma \ref{L - rank parity lemma} ensures that the rank of $\beta$ is permuted according to the choice $\beta(\pi)=v$ or $\beta(\pi)=w$. This shows that one can freely choose the rank  parity of $g$ by choosing an appropriate embedding of the above graph in the Moebius--Kantor graph $L_x$. 

This argument shows that for any face $g$ such that $g\cap \del B_n$ is a single edge $[x,y]$, the rank of both roots $\beta_x$ and $\beta_y$ corresponding to $x$ and $y$ in a 2-triangle containing $g$ can be chosen independently. In particular, the parity of $g$ can be chosen to be $p(g)$.  

In the last case, $f\cap \del B_n$ is a pair of adjacent edges intersecting at point $x$. In that case the ball $B_n$ determines a single edge of $L_x$, and therefore the parity of $f$ can be freely chosen.
\end{proof}

\begin{remark}
1) It seems clear from the proof that the parity alone will not determine the isomorphism class of a complex of rank \sq, and that there ought to exist uncountably many complexes with given parity (for example, uncountably many even/odd complexes of rank 7/4). Providing a formal proof of this assertion however requires a more refined invariant that distinguishes between complexes with a given parity map. We shall not pursue these investigations further. 

2) It follows from Theorem \ref{T - Free constructions} that the equivalence in Lemma \ref{L - e finite} fails without the assumption that the isometry group has compact quotient. 
 
\end{remark}

 The following should be compared to the results of \cite{tits1988spheres}. 
 
 \begin{corollary}
 There exists at least 174163 isomorphism types of spheres of radius 2 in complexes of rank \sq.
 \end{corollary}
   
   \begin{proof}
   A sphere of radius 2 determine the parity map on the ball $B$ of radius 1. By Theorem \ref{T - Free constructions}, for every map $p\colon B\to \ZI/2\ZI$ defined on the 2-skeleton, there exists a ball $B_p$ of radius 2 extending $B$, such that the parity of the faces of the ball of radius 1 in $B_p$ is given by $p$. Since the parity is an invariant, if $B_p$ and $B_{p'}$ are abstractly isomorphic, then there exists an automorphism $\theta\colon B\to B$ such that $p'=p\circ \theta$. Thus, the number of spheres of radius 2 is at least the number of $p$'s modulo the action of $\Aut(B)\simeq \Aut(G)$ where $G$ is the Moebius--Kantor graph. This gives at least $2^{24}/96=174762.6$ isomorphism types.    
   \end{proof}

\section{Generic constructions}\label{S - generic}

In \cite{E1,E2} we studied Euclidean buildings of type $\tilde A_2$ from a dynamical point of view, motivated by 
some questions in orbit equivalence theory. In view of the result in \ts\ref{S- Prescription}, it seems obvious that these results will carry over to complexes of rank \sq.  In the present section we mention an analog of  \cite[Theorem 5]{E1}, which can be stated as follows:

\begin{theorem}
A generic  complex of rank \sq\ has trivial automorphism group.
\end{theorem}

The notion of genericity used in this result (and in \cite{E1}) is topological, in the space of Baire, meaning that this holds for a dense $G_\delta$ set in the space of pointed complexes of rank \sq.

Rather than giving a formal proof of this result (which would significantly overlap with \cite{E1}), let us reproduce the figure from \cite[\ts 6]{E1} since it is helpful to understand how the generic triviality of the automorphism group occurs.

\begin{figure}[H]
\includegraphics[width=7cm]{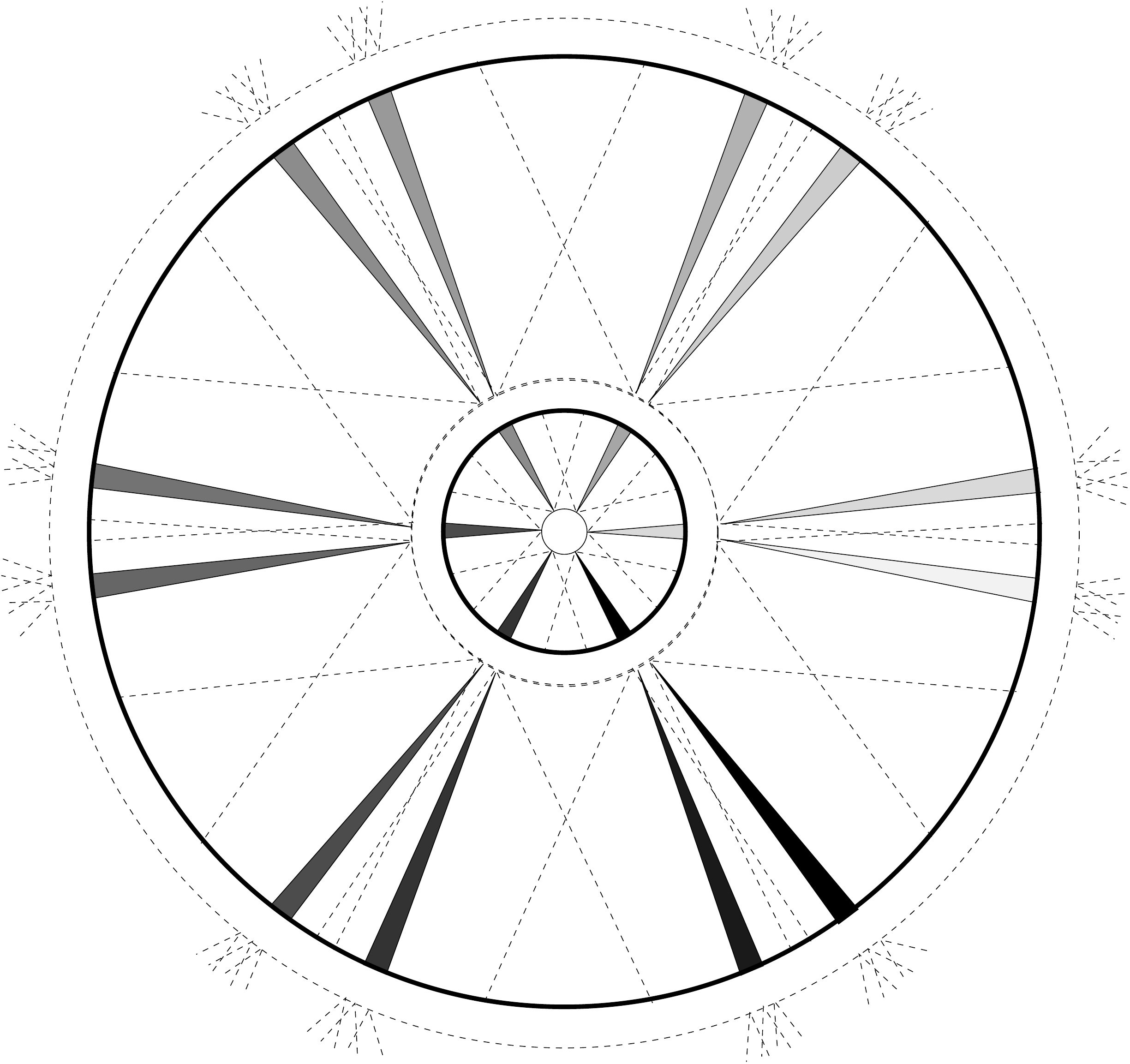}
\end{figure}

This figure symbolizes a complex of rank \sq\ (or a building in the situation of \cite{E1}). The central (white) disk represents a ball of ``small''   (i.e., arbitrary large, but fixed) radius, which corresponds to fixing a neighborhood in the space of pointed complexes of rank \sq. Above this ball are larger balls which we can control to trivialize  automorphisms, where the shades of grey account for the ``density'' of odd faces in the given portion of the complex. With some details, that need to be checked (as we did in \cite{E1}, in the case of buildings), this proves that the automorphism group of the generic (in the topological sense) space of rank \sq\ is trivial.

\begin{remark}
1) This argument is quite robust. It shows that the automorphism group of the spaces is generically trivial as soon as one can establish a free prescription theorem in the style of Ronan \cite{ronan1986construction} (or at least a partial prescription theorem which offers enough if not total freedom, as in \cite{E1} in the case $q\geq 3$), of a local invariant of isomorphism. The parity is such a local invariant of isomorphism, by Theorem \ref{T - Free constructions}. 

2) In the case of complexes of rank \sq, it is easy to check that the isotropy groups are finite, and therefore the full use of shadings is not required. Namely, one may trivialize the isotropy groups by using appropriate shades of grey in the first wreath, and then confine oneself to prescribe even wreaths, followed by odd ones in alternance.   

3)  The parity allows us to distinguish universal covers up to isometry, but due to its local nature, there seems to be no direct way to make use of it  for quasi-isometries.  It is of course natural to wonder what can be said about the relation of quasi-isometry in the space of complexes of rank \sq.
 
4) One motivation for studying the space of Euclidean buildings was, in the theory of orbit equivalence, to give new constructions of probability measure preserving standard equivalence relations with the property T of Kazhdan. One difficulty is the construction of diffuse invariant measures on this space. In \cite{surgery} it is shown that the (similarly defined) space of complexes of rank \sq, contrary to what we currently know in the case of Euclidean buildings,  admits diffuse invariant measures, albeit ones with amenable support. It would be interesting to study quasiperiodic spaces of rank \sq\ in more details, using the techniques of \cite{E1,E2} for example. 
\end{remark}

\section{Appendix}

In this appendix we compute the parity map for the remaining eight complexes of rank \sq\ from \cite{rd}, namely $V_1$, $V_3$, $V_5$,  $V_2^i$ for $i=1,\ldots,4$, and $V_4^2$.

\begin{table}[H]
\centering
\begin{tabular}{c@{\hspace{.1cm}} | c@{\hspace{.1cm}}  c@{\hspace{.1cm}} c@{\hspace{.1cm}}c@{\hspace{.1cm}}  c@{\hspace{.1cm}}  c@{\hspace{.1cm}}c@{\hspace{.1cm}} c@{\hspace{.1cm}}}
$V_1$& $[1,1,2]$&$[1,3,4]$&$[2,5,6]$&$[2,7,8]$&$[3,5,7]$&$[3,6,5]$&$[4,6,8]$&$[4,8,7]$\\
\hline
Parity&0&0&1&1&0&0&0&0 
\end{tabular}
\end{table}

\begin{table}[H]
\centering
\begin{tabular}{c@{\hspace{.1cm}} | c@{\hspace{.1cm}}  c@{\hspace{.1cm}} c@{\hspace{.1cm}}c@{\hspace{.1cm}}  c@{\hspace{.1cm}}  c@{\hspace{.1cm}}c@{\hspace{.1cm}} c@{\hspace{.1cm}}}
$V_2^1$& $[1,1,3]$&$[2,2,3]$&$[1,4,5]$&$[2,7,8]$&$[3,5,7]$&$[4,6,8]$&$[4,7,6]$&$[5,8,6]$\\
\hline
Parity&0&0&1&1&0&0&0&0 
\end{tabular}
\end{table}

\begin{table}[H]
\centering
\begin{tabular}{c@{\hspace{.1cm}} | c@{\hspace{.1cm}}  c@{\hspace{.1cm}} c@{\hspace{.1cm}}c@{\hspace{.1cm}}  c@{\hspace{.1cm}}  c@{\hspace{.1cm}}c@{\hspace{.1cm}} c@{\hspace{.1cm}}}
$V_2^2$& $[1,1,3]$&$[2,2,4]$&$[3,7,4]$&$[1,4,6]$&$[2,5,3]$&$[5,7,8]$&$[5,8,6]$&$[6,8,7]$\\
\hline
Parity&1&1&0&1&1&0&0&0 
\end{tabular}
\end{table}

\begin{table}[H]
\centering
\begin{tabular}{c@{\hspace{.1cm}} | c@{\hspace{.1cm}}  c@{\hspace{.1cm}} c@{\hspace{.1cm}}c@{\hspace{.1cm}}  c@{\hspace{.1cm}}  c@{\hspace{.1cm}}c@{\hspace{.1cm}} c@{\hspace{.1cm}}}
$V_2^3$& $[1,1,3]$&$[2,2,4]$&$[1,5,2]$&$[3,6,4]$&$[3,7,6]$&$[4,6,8]$&$[5,7,8]$&$[5,8,7]$\\
\hline
Parity&0&0&0&1&1&1&0&0 
\end{tabular}
\end{table}

\begin{table}[H]
\centering
\begin{tabular}{c@{\hspace{.1cm}} | c@{\hspace{.1cm}}  c@{\hspace{.1cm}} c@{\hspace{.1cm}}c@{\hspace{.1cm}}  c@{\hspace{.1cm}}  c@{\hspace{.1cm}}c@{\hspace{.1cm}} c@{\hspace{.1cm}}}
$V_2^4$& $[1,1,3]$&$[2,2,4]$&$[1,5,2]$&$[3,6,5]$&$[3,7,8]$&$[4,5,8]$&$[4,6,7]$&$[6,8,7]$\\
\hline
Parity&0&0&0&1&1&1&1&0 
\end{tabular}
\end{table}

\begin{table}[H]
\centering
\begin{tabular}{c@{\hspace{.1cm}} | c@{\hspace{.1cm}}  c@{\hspace{.1cm}} c@{\hspace{.1cm}}c@{\hspace{.1cm}}  c@{\hspace{.1cm}}  c@{\hspace{.1cm}}c@{\hspace{.1cm}} c@{\hspace{.1cm}}}
$V_3$& $[1,1,4]$&$[2,2,4]$&$[3,3,5]$&$[1,3,6]$&$[2,5,7]$&$[4,7,8]$&$[5,8,6]$&$[6,8,7]$\\
\hline
Parity&0&0&0&0&1&1&1&0 
\end{tabular}
\end{table}

\begin{table}[H]
\centering
\begin{tabular}{c@{\hspace{.1cm}} | c@{\hspace{.1cm}}  c@{\hspace{.1cm}} c@{\hspace{.1cm}}c@{\hspace{.1cm}}  c@{\hspace{.1cm}}  c@{\hspace{.1cm}}c@{\hspace{.1cm}} c@{\hspace{.1cm}}}
$V_4^2$& $[1,1,5]$&$[2,2,5]$&$[3,3,6]$&$[4,4,7]$&$[1,3,8]$&$[2,7,6]$&$[4,8,6]$&$[5,8,7]$\\
\hline
Parity&0&0&1&1&0&1&1&1
\end{tabular}
\end{table}

\begin{table}[H]
\centering
\begin{tabular}{c@{\hspace{.1cm}} | c@{\hspace{.1cm}}  c@{\hspace{.1cm}} c@{\hspace{.1cm}}c@{\hspace{.1cm}}  c@{\hspace{.1cm}}  c@{\hspace{.1cm}}c@{\hspace{.1cm}} c@{\hspace{.1cm}}}
$V_5$&  $[1,1,2]$&$[3,3,2]$&$[4,4,2]$&$[5,5,6]$&$[7,7,8]$&$[1,8,6]$&$[3,6,7]$&$[5,8,4]$\\
\hline
Parity&0&0&0&1&1&0&0&0 
\end{tabular}
\end{table}

\begin{remark}
 The complex $V_0$ is therefore the unique complex in the list of \cite{rd}  
such that every face is the centerpiece of a 2-triangle whose three sides are of rank 2.

\end{remark}

\end{document}